\documentclass{amsart}
\usepackage{amscd,graphicx,epsfig}
\usepackage[colorlinks=true, pdfstartview=FitV, linkcolor=blue, citecolor=blue, urlcolor=blue]{hyperref}
\usepackage{amssymb}

\usepackage[nohug]{diagrams}

\title[Endomorphisms]{Varieties Characterized by their Endomorphisms}
\author{Rafael B. Andrist and Hanspeter Kraft}
\date{August, 2013}

\address{Bergische Universit\"at Wuppertal,
Fachbereich C - Mathematik und Naturwissenschaften,
Gau{\ss}str. 20,
D-42119 Wuppertal}
\email{rafael.andrist@math.uni-wuppertal.de}

\address{Universit\"at Basel, Mathematisches Institut, Rheinsprung 21, CH-4051 Basel}
\email{hanspeter.kraft@unibas.ch}
\thanks{The second author was partially supported by SNF (Schweizerischer Nationalfonds)}

\newtheorem{thm}{Theorem}
\newtheorem*{thm*}{Theorem}

\newtheorem*{conj*}{Conjecture}

\newtheorem*{prob*}{Problem}
\newtheorem*{satz*}{Satz}

\newtheorem{prop}{Proposition}
\newtheorem{lem}{Lemma}
\newtheorem*{lem*}{Lemma}

\newtheorem*{cor*}{Corollary}

\theoremstyle{definition}
\newtheorem{defn}{Definition}

\theoremstyle{remark}
\newtheorem*{rem*}{Remark}
\newtheorem{rem}{Remark}

\newcommand{\op}{\operatorname}

\newcommand{\name}[1]{\textsc{#1\/}}
\newcommand{\NN}{{\mathbb N}}
\newcommand{\ZZ}{{\mathbb Z}}
\newcommand{\PP}{{\mathbb P}}
\newcommand{\FF}{{\mathbb F}}
\newcommand{\CC}{{\mathbb C}}

\renewcommand{\AA}{{\mathbb A}}

\newcommand{\Aone}{{\mathbb A}^{1}}
\newcommand{\Pone}{{\mathbb P}^{1}}

\newcommand{\VVV}{\mathcal V}
\newcommand{\simto}{\xrightarrow{\sim}}
\newcommand{\be}{\begin{enumerate}}
\newcommand{\ee}{\end{enumerate}}

\DeclareMathOperator{\Char}{char}
\DeclareMathOperator{\End}{End}
\DeclareMathOperator{\id}{id}

\DeclareMathOperator{\Mor}{Mor}
\DeclareMathOperator{\Aut}{Aut}
\DeclareMathOperator{\Spec}{Spec}

\newcommand{\bbmat}{\begin{bmatrix}}
\newcommand{\ebmat}{\end{bmatrix}}
\newcommand{\bsmat}{\begin{smallmatrix}}
\newcommand{\esmat}{\end{smallmatrix}}

\newcommand{\OOO}{\op{\mathcal O}}

\newcommand{\Zt}{{\tilde Z}}

\renewcommand{\phi}{\varphi}
\def \itt #1,#2:{\medskip\item[$\bullet$] %
     page\ \ignorespaces#1, line\ \ignorespaces#2:\ \ignorespaces}

\DeclareMathOperator{\CR}{CR}
\newcommand{\AutPone}{\Aut(\PP^{1})}
\newcommand{\Fpbar}{\overline{\FF_{p}}}

\newcommand{\margin}[1]{}
\newcommand{\lab}[1]{\label{#1}}

\begin{document}
\begin{abstract} We show that two varieties $X$ and $Y$ with isomorphic endomorphism semigroups are isomorphic up to field automorphism if one of them is affine and contains a copy of the affine line. A holomorphic version of this result is due to the first author.
\end{abstract}

\maketitle

\subsection*{Introduction}
It is a well-known fact that an affine variety $X$ over an algebraically closed field $k$ is determined (up to isomorphism) by its $k$-algebra of polynomial functions $\OOO(X)$. It is natural to ask whether other algebraic structures like the group of automorphisms or the semigroup of endomorphisms could determine a variety.

In general, the automorphism group might consist only of the identity, and the endomorphism semigroup might consist of the identity and the constant self-maps (see Proposition~\ref{endo-free.prop}). Considering only automorphisms is usually hopeless, and the situation is not even clear in the case of $\CC^n$ where the automorphism group is huge. The advantage of semigroups lies in the natural one-to-one correspondence between points of the variety and its constant maps. Moreover, any such isomorphism of semigroups $\Phi\colon \End(X) \to \End(Y)$ is induced by conjugation with a map $\phi\colon X \to Y$, i.e. $\Phi(f) = \varphi \circ f \circ \varphi^{-1}$ (see Remark~\ref{conjugating.rem}).

For topological spaces and continuous endomorphisms, the characterization by endomorphisms has been studied in detail, see e.g. the survey by \name{Magill} \cite{Ma1975A-survey}. 

For complex manifolds and holomorphic endomorphisms, the question has been investigated first by \name{Hinkkannen} in 1992 \cite{Hi1992Functions-Conjugating} who showed with elementary methods that a map $\phi \colon \CC \to \CC$ which conjugates endomorphisms of $\CC$ (i.e. entire holomorphic functions on $\CC$) to endomorphisms of  $\CC$ is a composition of a continuous field isomorphism and a holomorphic automorphism of $\CC$, i.e. $\phi(z) = az + b$ or $\phi(z) = a\overline{z} + b$ with $a, b \in \CC, a \neq  0$. 

In 1993, \name{Eremenko} \cite{Ere1993On-the-chara} proved for Riemann surfaces admitting non-constant bounded holomorphic functions that they are determined by their semigroup of endomorphisms. His result and method of proof was extended to bounded domains in $\CC^n$ by \name{Merenkov} \cite{Me2002Equivalence-of-Domains} in 2002 and to the case of $\CC^n$ by \name{Buzzard-Merenkov} \cite{BuMe2003Maps-Conjugating} in 2003.

The case of $\CC^n$ was generalized to Stein manifolds which contain a properly embedded copy of the affine complex line by the first author \cite{An2011Stein-spaces-chara} using a different method. The basic idea is to consider $\OOO(X)$ as a subset of $\End(X)$ with the help of the embedded affine line, and the crucial part is to identify the affine line after conjugation. This is possible due to the generalization of \name{Hartogs}' theorem on separate analyticity \cite{Ha1906Zur-theorie-der-ana} to complex Lie groups and making use of the fact that the affine line can be viewed as a Lie group.

\par\smallskip
We will solve here the problem in the algebraic setting, see Theorem \ref{thm1} below.

\begin{rem} 
If an affine variety $X$ admits a unipotent automorphism $u$ of infinite order, then it contains many affine lines, namely the non-trivial orbits under the action of the additive group $k^{+}:=\overline{\langle u \rangle}$. The existence of such actions is measured by the so-called \name{Makar-Limanov}-invariant, see \cite{Ma1996On-the-hypersurfac}.
\end{rem}

\subsection*{The main results}
From now on we assume that the base field $k$ is algebraically closed of arbitrary characteristic. We will confuse the affine line $\Aone$ with the field $k$ so that $\Aone$ has the structure of a field. In our setting, a variety $X$ is a $k$-variety, and $X$ is not necessarily irreducible.

\begin{thm}\lab{thm1}
Let $X$ and $Y$ be varieties and assume that there exists an isomorphism $\End(X)\simeq \End(Y)$ of semigroups. If $X$ is affine and contains a closed subvariety isomorphic to $\Aone$, then $X \simeq Y_{\sigma}$ where $\sigma$ is an automorphism of the base field $k$.
\end{thm}
\noindent
Here $Y_{\sigma}$ denotes the variety obtained from $Y$ by twisting with the morphism $\Spec \sigma^{-1}\colon\Spec k \to \Spec k$.
\par\smallskip
There is a canonical embedding $X \hookrightarrow \End(X)$ by sending $x\in X$ to the constant map $\gamma_{x}$ with value $x$. The following lemma is clear.
\begin{lem} The endomorphism $\gamma\in\End(X)$ is a constant map if and only if $\gamma\circ f = \gamma$ for all $f\in \End(X)$. Moreover,
we have $f\circ \gamma_{x}=\gamma_{f(x)}$ for any $f\in\End(X)$ and $x\in X$.
\end{lem}

\begin{rem}\lab{conjugating.rem}
The lemma implies that an isomorphism $\Phi\colon \End(X) \simto \End(Y)$ induces a bijection $\phi\colon X \to Y$, because $\Phi(\gamma_{x})$ is again a constant function, hence of the form $\gamma_{y}$ for some $y\in Y$. Moreover, we get $\Phi(f) = \phi\circ f \circ \phi^{-1}$ for all $f\in\End(X)$. In fact, this is clear for the constant functions, by definition of the map $\phi$, and then follows for all $f\in\End(X)$ from the second part of the lemma.
\end{rem}
This leads to the following definition.

\begin{defn} 
A map $\phi\colon X \to Y$ is called {\it conjugating\/} if it is bijective and induces a homomorphism $\Phi\colon\End(X)\to\End(Y)$ by $f\mapsto \phi\circ f \circ \phi^{-1}$. It is called {\it iso-conjugating\/} if, in addition, the induced homomorphism $\Phi\colon\End(X)\to\End(Y)$ is an isomorphism.
\end{defn}
Now the theorem above is a consequence of the next result.

\begin{thm}\lab{thm2}
Let $X,Y$ be two $k$-varieties and $\phi\colon X \to Y$ an iso-conjugating map. If $X$ is affine and contains a closed subvariety isomorphic to $\Aone$, then there is an automorphism $\sigma$ of $k$  such that the composition of $\phi$ with the canonical map $Y \to Y_{\sigma}$ is an isomorphism $X \simto Y_{\sigma}$ of varieties.
\end{thm}

\subsection*{Endo-free varieties}\lab{endo.subsec}
The following shows that the theorems do not hold if we drop the assumption that one of the varieties contains a copy of the affine line. 
Call a variety $X$ {\it endo-free\/} if the only endomorphisms are the identity and the constant maps. In this case we can identify $\End(X)$ with the semigroup  $\{\id\}\cup X$ where the multiplication on $X$ is given by $x\circ y = x$. This shows that every bijective map $\phi\colon X \to Y$ between two endo-free varieties is iso-conjugating. Hence Theorem 2 cannot hold for an endo-free variety. 

\begin{prop}\lab{endo-free.prop}
If $k$ is not the algebraic closure of a finite field, then there exist endo-free affine smooth varieties in any dimension.
\end{prop}
We thank the referee who pointed out to us that the assumption for $k$ is necessary. In fact, for $k=\overline{\FF_{p}}$ any $k$-variety is defined over a finite field and thus admits the Frobenius endomorphism. 

The proof of Proposition~\ref{endo-free.prop} is based on the following lemma which was communicated to us by  \name{J\'er\'emy Blanc}.

\begin{lem} \lab{endo.lem}
\be
\item Let $\phi\colon\PP^{1}\to \PP^{1}$ be a bijective morphism. 
If $\Char k = 0$, then $\phi$ is an isomorphism.
If $\Char k = p>0$, then $\phi=F^{m}\circ\alpha$ where $\alpha \in\Aut(\PP^{1})$ and $F$ is the Frobenius map.
\item
Let $\Delta_{1},\Delta_{2} \subset \PP^{1}$ be finite subsets with cardinality $|\Delta_{2}|\geq |\Delta_{1}|\geq 3$. Then every non-constant morphism $\phi\colon \PP^{1}\setminus\Delta_{1}\to \PP^{1}\setminus\Delta_{2}$ extends to a bijection $\tilde\phi\colon \PP^{1}\to\PP^{1}$ sending $\Delta_{1}$ to $\Delta_{2}$. In particular, $|\Delta_{1}|=|\Delta_{2}|$ if such a $\phi$ exists.
\ee
\end{lem}
\begin{proof}
(a) Composing a suitable $\beta\in\Aut(\PP^{1})$ with $\phi$ we can assume that $\phi':=\phi\circ\beta$ fixes $0,1$ and $\infty$. Hence $\phi' = \id$ if $\Char k=0$, and $\phi' = F^{m}$ if $\Char k =p>0$.

(b) We can extend $\phi$ to an endomorphism $\tilde\phi\colon \PP^{1}\to \PP^{1}$ where $\tilde\phi^{-1}(\Delta_{2})\subset\Delta_{1}$. Hence 
$|\Delta_{1}|= |\Delta_{2}|$, and $\tilde\phi$ restricts to a bijection $\Delta_{1}\to\Delta_{2}$. Choose two different points $p_{1},p_{2}\in \Delta_{1}$ and automorphisms $\alpha,\beta\in\Aut(\PP^{1})$ such that 
$$
\beta(0)=p_{1}, \quad\beta(\infty)=p_{2}, \quad \alpha(\tilde\phi(p_{1}))=0, \quad \alpha(\tilde\phi(p_{2}))=\infty.
$$
This implies that $\phi':=\alpha\circ\tilde\phi\circ\beta$ is of the form $x\mapsto ax^{s}$, and we can modify $\alpha$ so that $a=1$. If $s=1$, then $\tilde\phi=\alpha^{-1}\circ\beta^{-1}$ is an isomorphism. Otherwise, we choose $p_{3}\in\Delta_{1}\setminus\{p_{1},p_{2}\}$ and observe that the preimage of $\alpha(\tilde\phi(p_{3}))$ under $\phi'$ is a single point. This implies that $s$ is a power of $p$, hence $\tilde\phi$ is bijective.
\end{proof}

\begin{proof}[Proof of Proposition~\ref{endo-free.prop}]
We first assume that $\Char k = p>0$.
Let $m\geq 4$ and define
$$
U_{m}:=\{\Delta \subset \PP^{1}\mid |\Delta|=m \text{ and } \{0,1,\infty\} \subset \Delta\}.
$$
$U_{m}$ is a closed subset of the symmetric product $S^{m}\PP^{1}$ consisting of all unordered subsets of $\PP^{1}$ of cardinality $m$. We will also need the notation $U_{m}(\Fpbar)$ for the subset of those $\Delta\in U_{m}$ such that $\Delta \subset \Pone(\Fpbar)$.

(a) Choose $\Delta_{0} \in U_{m}$. Since an automorphism of $\PP^{1}$ is determined by the images of $0$, $1$ and $\infty$ it follows that there are only finitely many $g\in\Aut(\PP^{1})$ such that $g(\Delta_{0})\in U_{m}$:
$$
\AutPone\,\Delta_{0}\cap U_{m} = \{g_{0}(\Delta_{0}),g_{1}(\Delta_{0}),\ldots,g_{k}(\Delta_{0})\}.
$$
It now follows from Lemma~\ref{endo.lem} that for $\Delta\in U_{m}\setminus \{F^{j}g_{i}\Delta_{0}\mid j\in\ZZ, \ i=0,\ldots,k\}$, every morphism $\PP^{1}\setminus \Delta_{0} \to \PP^{1}\setminus\Delta$ is constant, as well as every morphism $\PP^{1}\setminus \Delta \to \PP^{1}\setminus\Delta_{0}$. 

\par\smallskip
(b) If $\Delta\in U_{m}$ and $\Delta\not\subset \PP^{1}(\overline{\FF_{p}})$, then every non-constant endomorphism $\phi$ of $\PP^{1}\setminus\Delta$ is an automorphism. In fact, by Lemma~\ref{endo.lem}(b), the endomorphism extends to a bijective  endomorphism of $\PP^{1}$ permuting $\Delta$,  and thus a suitable power  $\phi^{n}$ is the identity on $\Delta$ which implies that $\phi^{n}=F^{r}$ for some $r\in\NN$. Because $\Delta$ contains elements which are not fixed by any power of $F$, we must have $r=0$.

\par\smallskip
(c) Recall that the cross-ratio $\CR(p_{1},p_{2},p_{3},p_{4})$ of an ordered 4-tuple of points of $\PP^{1}$ is a projective invariant.
Denote by $U_{m}'\subset U_{m}$ the open dense subset of those $\Delta$ which satisfy the following condition:
\be
\item[(CR)] If $\{p_{1},p_{2},p_{3},p_{4}\}$ and $\{q_{1},q_{2},q_{3},q_{4}\}$ are two different subsets of 4 points of $\Delta$, then $\CR(p_{i_{1}},p_{i_{2}},p_{i_{3}},p_{i_{4}})\neq\CR(q_{j_{1}},q_{j_{2}},q_{j_{3}},q_{j_{4}})$ for every ordering of the $p_{i}$ and the $q_{j}$.
\ee
It follows that if $m\geq 5$ and  $\Delta\in U_{m}'$, then every automorphism of $\PP^{1}\setminus\Delta$ is the identity. In fact, the automorphism extends to $\PP^{1}$ and permutes the elements of $\Delta$. Since the cross-ratio is preserved under an automorphism, this permutation of $\Delta$  preserves all subsets of 4 points, hence is the identity, because $|\Delta|>4$.
\par\smallskip
(d) Finally, we claim that for $m\geq 5$ and $d\in\NN$ we can find $\Delta_{1},\ldots,\Delta_{d} \in U_{m}'$ such that the following holds:
\be
\item[(i)] $\Delta_{j}\not\subseteq \Pone(\Fpbar)$ for all $j$;
\item[(ii)] There are no non-constant morphisms $\PP^{1}\setminus \Delta_{j}\to \PP^{1}\setminus \Delta_{k}$ for $j\neq k$.
\ee
It then follows from (c) and (d)  that the $d$-dimensional affine variety 
$$
(\PP^{1}\setminus\Delta_{1})\times \cdots \times (\PP^{1}\setminus\Delta_{d})
$$ 
is smooth and endo-free.

In order to prove the claim we proceed by induction on $d$. The case $d=1$ is clear, because condition (ii) is empty. Assume that $\Delta_{1},\ldots,\Delta_{d}\in U_{m}'$ satisfy (i) and (ii). We have to construct $\Delta\in U_{d}'\setminus U_{d}'(\Fpbar)$ such that there are no non-constant morphisms $\PP^{1}\setminus\Delta_{j}\to \PP^{1}\setminus\Delta$ and  $\PP^{1}\setminus\Delta\to \PP^{1}\setminus\Delta_{j}$ for $j=1,\ldots,d$. 

We have seen in (a) that this is the case if $\Delta$ does not belong to a finite set of $F$-orbits $\{F^{j}\Sigma\mid j\in\ZZ\}$ for certain $\Sigma \in U_{d}'$. But such a $\Delta$ always exists. In fact, the subset $k \setminus\Fpbar\subset k$ is clearly not contained in a finite set of $F$-orbits $\{F^{j}a\mid j\in \ZZ\}$ of elements $a \in k$, and this carries over to $U_{m}'\setminus U_{m}'(\Fpbar)$ by looking at the $F$-equivariant map $U_{m}\to k$, $\{0,1,\infty,p_{4},\ldots,p_{m}\} \mapsto p_{4}+\cdots+p_{m}$, whose restriction to $U_{m}'\setminus U_{m}'(\Fpbar)$ contains $k\setminus\Fpbar$ in its image, because $m\geq 5$. This completes the proof of the proposition in $\Char k > 0$.
\par\smallskip
The proof in characteristic zero is similar, but much simpler. In fact, it follows from (c) that the curves $\Pone\setminus\Delta$ are endo-free if $\Delta\in U_{m}'$, and with the notation from (a) we see that all morphisms $\PP^{1}\setminus \Delta_{0} \to \PP^{1}\setminus\Delta$ and $\PP^{1}\setminus \Delta \to \PP^{1}\setminus\Delta_{0}$ are constant if $\Delta$ does not belong to the finite set $\AutPone\Delta_{0}\cap U_{m}$. So it is clear that we can find $\Delta_{1},\ldots,\Delta_{d} \in U_{m}'$ such that there are no non-constant morphisms $\PP^{1}\setminus \Delta_{j}\to \PP^{1}\setminus \Delta_{k}$ for $j\neq k$, and we are done.
\end{proof}

\subsection*{Iso-conjugating maps}
For $S \subset \End(X)$ and $x_{0}\in X$ we define the following {\it zero set\/}:
$$
\VVV(S,x_{0}):=\{x \in X \mid f(x) = x_{0}\text{ for all } f\in S\} \subset X.
$$
This is clearly a closed subset of $X$.
\begin{lem}\lab{closedimage.lem}
If $\phi\colon X \to Y$ is a conjugating map and $\Phi\colon \End(X)\to \End(Y)$ the corresponding homomorphism, then $\phi(\VVV(S,x_{0})) = \VVV(\Phi(S),\phi(x_{0}))$.
\end{lem}
\begin{proof} We have
\begin{eqnarray*}
x \in \VVV(S,x_{0}) &\iff& f(x) = x_{0} \text{ for all } f\in S \\
&\iff& \phi(f(x)) = \phi(x_{0}) \text{ for all } f\in S \\
&\iff& \Phi(f)(\phi(x)) = \phi(x_{0}) \text{ for all } f\in S \\
&\iff& \phi(x)\in \VVV(\Phi(S),\phi(x_{0}))
\end{eqnarray*}
\end{proof}

\begin{rem}\lab{closed.rem}
In general, not every closed subset $A \subset X$ is of the form $\VVV(S,x_{0})$. However, if $X$ is affine and contains a copy of the affine line, then we have an embedding $\OOO(X)=\Mor(X,\Aone) \subset \End(X)$,  and one gets  $\VVV_{X}(S) = \VVV(S,0)$ for $S\subset \OOO(X)$ and $0\in\Aone\subset X$ where $\VVV_{X}(S) \subset X$ is the zero set of $S$. Hence in this case, every closed subset $A \subset X$ is of the form $\VVV(S,x_{0})$.
\end{rem}

\begin{lem}\label{lem2}\lab{indconjmap.lem}
Let $\phi\colon X \to Y$ be a conjugating map where $X$ is affine. If $A \subset X$ is a closed subset isomorphic to $\Aone$, then $\phi(A) \subset Y$ is closed and the induced map $\phi|_{A}\colon A \to \phi(A)$ is conjugating. 
\end{lem}
\begin{proof}
By the previous remark and Lemma~\ref{closedimage.lem} the subset $B:=\phi(A)\subset Y$ is closed.
Moreover, the restriction $\OOO(X) \to \OOO(A)$ is surjective. This implies that $\Mor(X,A) \to \Mor(A,A)=\End(A)$ is surjective, because $A \simeq \Aone$. If $\alpha\in\End(A)$ and $\tilde\alpha\colon X \to A$ a lift of $\alpha$, then $\Phi(\tilde\alpha)=\phi\circ \tilde\alpha\circ \phi^{-1}\in\End(Y)$ maps $B$ into itself, and $\Phi(\tilde\alpha)|_{B} = \phi|_{A}\circ \alpha \circ (\phi|_{A})^{-1}\in\End(B)$.
\end{proof}

\subsection*{Algebraic fields}\lab{fields.subsec}
The basic ingredient in the proof of Theorem~\ref{thm2} is the following result.

\begin{prop}\label{basic.prop}
Let $\phi\colon \AA^{1} \to Z$ be a conjugating map. Then there is an isomorphism $\psi\colon Z \simto \Aone$ such that the composition $\psi\circ\phi\colon \Aone \to \Aone$ is a field automorphism.
\end{prop}

\begin{proof}
The variety $Z$ inherits from $\Aone$ the structure of a commutative field isomorphic to $k$ such that the additions $\alpha_{z_{0}}\colon z\mapsto z+z_{0}$, the multiplications $\mu_{z_{0}}\colon z\mapsto z_{0}z$ and the power maps $z\mapsto z^{n}$, $n\in\NN$, are morphisms. It follows that $Z$ is smooth.

Let $0$ resp. $1 \in Z$ be the identity elements for addition resp. multiplication, and put $Z^{*}:= Z \setminus\{0\}$.

(a) $Z$ is connected, hence irreducible. In fact, if $Z^{0}$ denotes the connected component of $0$, then any other component has the form $z_0+ Z^{0}$. Since $0\in Z^{0}$ and $0 = z_{0}\,0 \in z_{0}Z^{0}$ it follows that  $z_{0} Z^{0}= Z^{0}$ for all $z_{0}\in Z^{*}$.  Hence $Z = Z^{0}$, because $Z^{0}$ contains elements $\neq 0$.

(b) Let $\ell$ be a prime different from $\Char k$, and denote by $\lambda\colon Z \to Z$ the power map $z\mapsto z^{\ell}$. Put $K:=\lambda^{*}(k(Z)) \subseteq k(Z)$ and denote by $K_{s}\subset k(Z)$ the separable closure of $K$ in $k(Z)$. 
Let $\Zt$ be the normal closure of $Z$ in $K_{s}$. Then we get the following commutative diagram
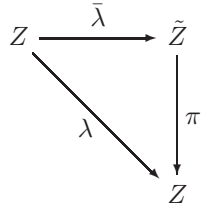
\begin{diagram}
Z & \rTo^{\bar \lambda}& \Zt\\
& \rdTo_{\lambda} & \dTo_{\pi} \\
&& Z 
\end{diagram}
where $\bar \lambda$ is finite and purely inseparable, hence bijective, and $\pi$ is finite and separable. Moreover,  the multiplications $\mu_{z}$ induce automorphisms $\bar\mu_{z}$ of $\Zt$ such that 
the following diagram commutes:
\begin{diagram}
Z & \rTo^{\mu_{z}} &  Z  & \lTo^{\supset} & Z^{*}\\
\dTo<{\bar\lambda} &&  \dTo<{\bar\lambda}>{\text{\tiny bijective}} && \dTo<{\bar\lambda}\\
\Zt & \rTo^{\bar\mu_{z}} & \Zt & \lTo^{\supset} & \Zt^{*}:=\Zt \setminus\bar \lambda(0)\\
\dTo<{\pi} &&  \dTo<{\pi} && \dTo<{\pi}\\
Z & \rTo^{\mu_{z^{\ell}}} & Z & \lTo^{\supset} & Z^{*}\\
\end{diagram}
As a consequence, $\pi\colon \Zt^{*} \to Z^{*}$ is smooth,  
because $\pi$ is smooth on a non-empty open set $U\subset \Zt$ and therefore smooth in $\bar\mu_{z}(U)$ for all $z\in Z^{*}$. 

(c) The map $\pi\colon \Zt \to Z$ is ramified in $\bar 0:=\bar \lambda(0)\in\Zt$. In fact, consider the multiplication $\mu_{\zeta}$ where $\zeta$ is a primitive $\ell$th root of unity. From above, we get the following commutative diagram:
\begin{diagram}
\Zt  &    \rTo^{\bar\mu_{\zeta}}     &            \Zt          \\
   \dTo<{\pi}  &&    \dTo>{\pi}        \\
 Z    &         \rTo^{\id}             &   Z         \\
\end{diagram}
The differential of $\bar\mu_{\zeta}$ in $\bar 0$ is a non-trivial automorphism of the tangent space $T_{\bar 0}\Zt$, because the order of $\mu_{\zeta}$ is prime to $\Char(k)$.  Hence, $d\pi_{\bar 0}\colon T_{\bar 0}\Zt \to T_{0}Z$ cannot be an isomorphism. 

(d) Now the ``purity of the branch locus'' implies that $\dim Z = 1$ (cf. \cite{AlKl1971On-the-purity-of-t}, \cite{AlKl1973Correction-to:-On-}). Hence $Z$ is either an affine or a projective curve. In both cases, $Z^{*}$ is a smooth affine algebraic curve whose automorphism group is infinite,  and so $Z^{*}\simeq \Aone\setminus\{0\}$ or $Z^{*}\simeq\Aone$. In the second case, $\Zt \simeq \PP^{1}$ which is impossible, because $\PP^{1}$ has no automorphism of order $\ell$ with a single fixed point. Thus $Z \simeq \Aone$, and there is a unique isomorphism $\psi\colon Z\simto \Aone$ with $\psi(0)=0$ and $\psi(1)=1$. It follows that $\rho:=\psi\circ\phi\colon \Aone \to \Aone$ is a conjugating map with $\rho(0)=0$ and $\rho(1)=1$. 

(e) It remains to see that $\rho$ is a field automorphism. The map $\rho\circ\mu_{z}\circ\rho^{-1}$ is an automorphism of $\Aone$ fixing $0$ and sending  $1$ to $\rho(z)$, hence it is equal to  $\mu_{\rho(z)}$. It follows that 
$$
\rho(z_{1}z_{2}) = \rho(\mu_{z_{1}}(z_{2})) = (\rho\circ\mu_{z_{1}}\circ\rho^{-1})(\rho(z_{2})) = \mu_{\rho(z_{1})}(\rho(z_{2})) = \rho(z_{1})\rho(z_{2}).
$$
Similarly, we see that $\rho\circ\alpha_{z}\circ\rho^{-1}$ is a fixed point free automorphism of $\Aone$ sending $0$ to $\rho(z)$, hence  is equal to $\alpha_{\rho(z)}$ which implies, as before, that $\rho(z_{1}+z_{2}) = \rho(z_{1})+\rho(z_{2})$.
\end{proof}

\begin{rem} In case of $k = \CC$  it suffices to consider only the additive structure on $Z$. A result from \name{Palais}' \cite{Pa1978Some-analogues-of-} on separately polynomial maps then shows that $Z$ has the structure of an algebraic group. Using the Euclidean topology on $Z$ and the universal covering $\CC^d \to Z$ one could conclude, as in the holomorphic setting (see \cite{An2011Stein-spaces-chara}), that $Z$ is isomorphic to $\CC$. 
\end{rem}

\subsection*{Proof of the main result}\lab{proof.sec}
For a variety $Y$ and a field automorphism $\sigma\colon k \to k$ we define $Y_{\sigma}:=\Spec k \times_{\Spec k}Y$ using the base change $\Spec\sigma^{-1}\colon \Spec k \to \Spec k$. 
Thus, we have a canonical bijection $\pi_{\sigma}\colon Y_{\sigma}\to Y$ which sends closed sets into closed sets and induces an isomorphism $\End(Y_{\sigma})\simto\End(Y)$, i.e. $\pi_{\sigma}$ is iso-conjugating. If $Y \subset k^{n}$ is defined by the polynomials $f_{1},\ldots,f_{m}$, then $Y_{\sigma} \subset k^{n}$ is defined by the polynomials $f_{1}^{\sigma},\ldots,f_{m}^{\sigma}$ where $(\sum_{i} c_{i}x^{i})^{\sigma}:=\sum_{i}\sigma^{-1}(c_{i}) x^{i}$, and the map $\pi_{\sigma}\colon Y_{\sigma}\to Y$ is given by $(a_{1},\ldots,a_{n})\mapsto (\sigma a_{1},\ldots,\sigma a_{n})$. In particular, $(\AA^{n})^{\sigma}=\AA^{n}$ and $\pi_{\sigma}=\sigma$ in this case.

\begin{proof}[Proof of Theorem~\ref{thm2}] We fix a closed embedding $\Aone\subseteq X$. Since $X$ is affine, the image of $\Aone$ is a zero set of the form $\VVV(S,0)$ and so $Z:=\phi(\Aone) \subset Y$ is closed and the induced map $\Aone\to Z$ is conjugating (Lemma~\ref{lem2}). It follows from Proposition~\ref{basic.prop} that  there is an isomorphism $\psi\colon Z\to\Aone$ such that the composition $\sigma:=\psi\circ\phi\colon \Aone \to \Aone$ is a field automorphism. Composing $\phi$ with $\pi_{\sigma}^{-1}\colon Y \to Y_{\sigma}$ and replacing $Y$ by $Y_{\sigma}$ we can assume that there is an embedding $\Aone\subseteq Y$ such that $\phi|_{\Aone}=\id_{\Aone}$:
$$
\begin{CD}
X   @>\phi>>  Y \\
@A{\cup}AA  @A{\cup}AA \\
\Aone @= \Aone 
\end{CD}
$$
We thus obtain an embedding $\OOO(X)=\Mor(X,\Aone) \subset \End(X)$, and similarly for $Y$. Moreover, for $f\in\OOO(Y)$, we find 
$$
\phi^{*}(f) =  f \circ \phi = \phi^{-1}\circ f \circ \phi \in \Mor(X,\Aone)=\OOO(X).
$$
The same holds for $(\phi^{-1})^{*}$ which shows that $\phi^{*}\colon\OOO(Y) \to \OOO(X)$ is an isomorphism. Since $X$ is affine, it  follows that $\psi:=\phi^{-1}\colon Y \to X$ is a morphism inducing an isomorphism $\psi^{*}\colon \OOO(X)\simto \OOO(Y)$. Now the claim follows from  Lemma~\ref{final.lem} below.
\end{proof}

\begin{lem}\label{final.lem}
Let $Y$ be a variety. Assume that $\OOO(Y)$ is finitely generated and that the canonical morphism $\psi_{Y}\colon Y \to \Spec\OOO(Y)$ is bijective. 
Then $Y$ is affine and $\psi_{Y}$ an isomorphism.
\end{lem}
\begin{proof} 
\name{Zariski}'s Main Theorem in \name{Grothendieck}'s form (\cite[Th\'eor\`eme 8.12.6]{Gr1967Elements-de-geomet}, cf. \cite[Chap.~III, \S9]{Mu1999The-red-book-of-va}) implies that  $\psi_{Y}$ admits a factorization $\psi_{Y}=\tau\circ\eta$,
\begin{diagram}
Y  &    \rInto^{\eta}     &            Y'                    \\
     &  \rdTo<{\psi_{Y}}  &    \dTo>{\tau}        \\
     &                       &    \Spec\OOO(Y)         \\
\end{diagram}
where $\eta$ is an open immersion and $\tau$ a finite morphism. From this we get inclusions $\OOO(Y) \subseteq \OOO(Y')\subseteq\OOO(Y)$, 
and thus $\tau$ is an isomorphism. It follows that $\psi_{Y}$ is a bijective open immersion, hence an isomorphism,
too.
\end{proof}

\vskip1cm

\begin{thebibliography}{Mum99}

\bibitem[AK71]{AlKl1971On-the-purity-of-t}
Allen Altman and Steven~L. Kleiman, \emph{On the purity of the branch locus},
  Compositio Math. \textbf{23} (1971), 461--465.

\bibitem[AK73]{AlKl1973Correction-to:-On-}
\bysame, \emph{Correction to: ``{O}n the purity
  of the branch locus'' ({C}ompositio {M}ath. {\bf 23} (1971), 461--465)},
  Compositio Math. \textbf{26} (1973), 175--180.

\bibitem[And11]{An2011Stein-spaces-chara}
Rafael~B. Andrist, \emph{Stein spaces characterized by their endomorphisms},
  Trans. Amer. Math. Soc. \textbf{363} (2011), no.~5, 2341--2355.

  
\bibitem[BM03]{BuMe2003Maps-Conjugating}
Gregery~T. Buzzard and Sergei Merenkov: \emph{Maps Conjugating Holomorphic Maps in $\CC^n$.} Indiana Univ. Math. J., \textbf{52} No. 5: (2003) 1135--1146.

%
\bibitem[Gro67]{Gr1967Elements-de-geomet}
Alexander~Grothendieck, \emph{\'{E}l\'ements de g\'eom\'etrie alg\'ebrique. {IV}.
  \'{E}tude locale des sch\'emas et des morphismes de sch\'emas {(Troisi\`eme Partie)}}, Inst.
  Hautes \'Etudes Sci. Publ. Math. (1966), no.~28.

\bibitem[Har06]{Ha1906Zur-theorie-der-ana}
Fritz Hartogs, \emph{Zur Theorie der analytischen Funktionen mehrerer unabh{\"a}ngiger Ver{\"a}nderlichen, insbesondere {\"u}ber die Darstellung derselben durch Reihen, welche nach Potenzen einer Ver{\"a}nderlichen fortschreiten}, Math. Ann. \textbf{62} (1906), 1--88.

\bibitem[Hin92]{Hi1992Functions-Conjugating}
Aimo Hinkkanen: \emph{Functions Conjugating Entire Functions to Entire Functions and Semigroups of Analytic Endomorphisms.} Complex Variables Theory Applications \textbf{18} (1992), 149--154.

\bibitem[Ere93]{Ere1993On-the-chara}
Alexandre Eremenko: \emph{On the characterization of a Riemann surface by its semigroup of endomorphisms.} Trans. of AMS \textbf{338} No. 1 (1993), 123--131.

\bibitem[Mag75]{Ma1975A-survey}
Kenneth~D. Magill~Jr.: \emph{A survey of semigroups of continuous self-maps.} Semigroup Forum \textbf{11} (1975/76), 189--282.

\bibitem[ML96]{Ma1996On-the-hypersurfac}
Leonid~Makar-Limanov, \emph{On the hypersurface {$x+x^2y+z^2+t^3=0$} in {${\bf
  C}^4$} or a {${\bf C}^3$}-like threefold which is not {${\bf C}^3$}}, Israel
  J. Math. \textbf{96} (1996), 419--429.

\bibitem[Mer02]{Me2002Equivalence-of-Domains}
Sergei Merenkov: \emph{Equivalence of Domains with Isomorphic Semigroups of Endomorphisms.} Proc. of AMS \textbf{130}, No. 6 (2002), 1743--1753
  
%
\bibitem[Mum99]{Mu1999The-red-book-of-va}
David Mumford, \emph{The red book of varieties and schemes}, expanded ed.,
  Lecture Notes in Mathematics, vol. 1358, Springer-Verlag, Berlin, 1999,
  Includes the Michigan lectures (1974) on curves and their Jacobians, With
  contributions by Enrico Arbarello. 

\bibitem[Pal78]{Pa1978Some-analogues-of-}
Richard~S. Palais, \emph{Some analogues of {H}artogs' theorem in an algebraic
  setting}, Amer. J. Math. \textbf{100} (1978), no.~2, 387--405.
  
%

  
\end{thebibliography}

\end{document}